\documentclass[12pt]{amsart}
\usepackage{a4wide}
\usepackage{amssymb}
\usepackage{amsmath}
\usepackage{amsfonts}
\usepackage{amsthm}
\usepackage{cite}
\usepackage{graphicx}
\usepackage{epsfig}
\usepackage{longtable}
\usepackage{paralist}
\usepackage[usenames,dvipsnames]{color}
\usepackage{hyperref}
\usepackage{caption}
\usepackage[labelformat=simple,labelfont={}]{subfig}

\newtheorem{thm}{Theorem}

\newtheorem{prop}[thm]{Proposition}

\theoremstyle{definition}

\newtheorem*{ackno}{Acknowledgement}

\theoremstyle{remark}
\newtheorem{rmk}{Remark}

\newcommand{\DEF}{{:=}}

\DeclareMathOperator{\dd}{\mathrm{d}}

\DeclareMathOperator{\gammafcn}{\Gamma}

\DeclareMathOperator{\GegenbauerC}{\mathrm{C}}

\DeclareMathOperator{\digammafcn}{\psi}

\DeclareMathOperator{\HyperF}{F}

\newcommand{\HyperpFq}[2]{_{#1}\!\HyperF\!_{#2}}
\newcommand{\Hypergeom}[5]{{\sideset{_#1}{_#2}\HyperF\!\left(
      \genfrac{}{}{0pt}{0}{#3}{#4}
      \middle|\,#5\right)}}

\newcommand{\Pochhsymb}[2]{{\left(#1\right)_{#2}}}

\allowdisplaybreaks[1]

\title{Weighted $L^2$-norms of Gegenbauer
  polynomials} \author[J. S. Brauchart and P. J. Grabner]{Johann S. Brauchart
  and Peter J. Grabner\textsuperscript{\textasteriskcentered{}\textdagger{}}}
\thanks{\noindent
  \textsuperscript{\textasteriskcentered} Corresponding author. \\
  \textsuperscript{\textdagger} The research of this author was supported by
  the Austrian Science Fund FWF project F5503 (part of the Special Research
  Program (SFB) ``Quasi-Monte Carlo Methods: Theory and Applications'').}

\date{\today}

\DeclareCaptionLabelFormat{andtable}{#1˜\ #2 and \tablename˜\ \thetable}

\begin{document}

\address{J. S. Brauchart, P. J. Grabner: 
Institute of Analysis and Number Theory, 
Graz University of Technology, 
Kopernikusgasse 24/II, 
8010 Graz, 
Austria}
\email{j.brauchart@tugraz.at, peter.grabner@tugraz.at}

\begin{abstract}
  We study  integrals of the form
  \begin{equation*}
    \int_{-1}^1(C_n^{(\lambda)}(x))^2(1-x)^\alpha (1+x)^\beta\dd x,
  \end{equation*}
  where $C_n^{(\lambda)}$ denotes the Gegenbauer-polynomial of index
  $\lambda>0$ and $\alpha,\beta>-1$. We give exact formulas for the integrals
  and their generating functions, and obtain asymptotic formulas as
  $n\to\infty$.
\end{abstract}
\date{\today}

\maketitle
\section{Introduction}\label{sec:introduction}
Integrals of the form
\begin{equation}\label{eq:int_p2}
  \int_I p_n^2(x)\,w(x)\dd x,
\end{equation}
where $(p_n)_{n\in\mathbb{N}_0}$ is a sequence of orthogonal polynomials with
respect to some weight $\widetilde{w}$ on the interval $I$ (see
\cite{Szegoe1939:orthogonal_polynomials}), have occurred in different
context. Of course, the case when $w\neq\widetilde{w}$ is the interesting one.

Such integrals for Legendre, associated Legendre and Gegenbauer polynomials
occur in explicit computations of angular momentum in classical as well as
quantum mechanics (see \cite{Edmonds1957:angular_momentum}). Based on this
interest in these computations there exists an extensive literature in a
physics context (see for instance
\cite{Rashid1986:evaluation_integrals_involving,
  Samaddar1974:some_integrals_involving,
  Ullah1984:evaluation_integral_involving,
  Laursen_Mita1981:some_integrals_involving}).

Determinantal point processes (see
\cite{Hough_Krishnapur_Peres+2009:zeros_gaussian_analytic}) have been
introduced also with a strong motivation from physics; they are used to model
Fermionic particles. Since then they have become the object of mathematical
research from various perspectives. One aspect that makes these processes
interesting is their built-in repulsion between different point, which amounts
in better distribution properties of the sample points as compared to
i.~i.~d. points. Also, as a special feature of these processes the computation
of expectations of discrete energy expressions (for a comprehensive introduction and collection of recent results see \cite{Borodachov_Hardin_Saff2019})
\begin{equation*}
  \sum_{i\neq j}f(\|x_i-x_j\|)
\end{equation*}
is computationally feasible. Here $f$ is some potential depending only on the
distance of two points. In many cases these computations lead to integrals of
the form \eqref{eq:int_p2} (see
\cite{Brauchart_Grabner_Kusner+2020:hyperuniform_point_sets,
  Beltran_Ferizovic2020:approximation_to_uniform,
  Beltran_Marzo_Ortega-Cerda2016:determinantal,AlishahiZamani}).

A further probabilistic model that yields to the study of integrals of the form
\eqref{eq:int_p2} has been studied in
\cite{Cammarota_Marinucci2018:quantitative_central_limit}. Here the Gaussian
random field on the sphere $\mathbb{S}^2$ given by
\begin{equation*}
  f_\ell(x)=\sqrt{\frac{4\pi}{2\ell+1}}\sum_{m=-\ell}^\ell a_{\ell m}Y_{\ell m}(x)
\end{equation*}
is studied. Here $(Y_{\ell m})_{\ell=-m}^m$ is an orthonormal base of the space
of spherical harmonics of degree $\ell$ (see
\cite{Mueller1966:spherical_harmonics}) and $(a_{\ell m})_{\ell=-m}^m$ are
independent Gaussian random variables with mean $0$ and variance $1$. Then the
asymptotic study of the distribution of the Euler-Poincar\'e characteristic of
the random field $f_\ell$ involves \emph{inter alia} integrals of the form
\eqref{eq:int_p2}.

We took this as a motivation to provide a general study of such integrals,
where $(p_n)_n$ are Gegenbauer polynomials, and $w(x)$ are Gegenbauer or Jacobi
weights. A special case has been studied in \cite{Ferizovic2020:l_2_norm}.

{\bf Outline of the paper.} In Section~\ref{sec:notation}, we provide notations and collect frequently used facts. In Section~\ref{sec:expl-form-gener}, we define the integral and present explicit formulas in the most general case of Jacobi weights and give the generating function relation. Section~\ref{sec:singularity-analysis} gives a brief introduction into the method of singularity analysis and Section~\ref{sec:mell-barn-form} provides the Mellin-Barnes integral representations of the generating functions for Jacobi and Gegenbauer weights. Section~\ref{sec:jacobi-weights} discusses the generic case for the Jacobi weight. Main results are the asymptotic series relation \eqref{eq:In-asymp} with explicit coefficients and Theorem~\ref{thm:jacobi} concerning the asymptotic leading term. Section~\ref{sec:gegenbauer-weights} discusses the generic case for Gegenbauer weights. Section~\ref{sec:special-cases} provides connection formulas for the integrals and selected non-generic cases.

\section{Preliminaries}\label{sec:notation}
Throughout this paper we use the Gegenbauer polynomials with their standard
normalisation (see
\cite{Magnus_Oberhettinger_Soni1966:formulas_theorems_special}) given by
\begin{equation*}
  \sum_{n=0}^\infty \GegenbauerC_n^{(\lambda)}(x)z^n=\frac1{(1-2xz+z^2)^\lambda}.
\end{equation*}
These polynomials are orthogonal with respect to the weight function
$(1-x^2)^{\lambda-\frac12}$ on the interval $[-1,1]$ and normalised such that (see for
instance \cite{Andrews_Askey_Roy1999:special_functions}) 
\begin{equation} \label{eq:Cnlambda.of.one}
\GegenbauerC_n^{(\lambda)}( 1 ) = \frac{\Pochhsymb{2\lambda}{n}}{n!} = \frac{1}{\gammafcn( 2 \lambda )} \, \frac{\gammafcn( n + 2\lambda )}{\gammafcn( n + 1 )} \sim \frac{1}{\gammafcn( 2 \lambda )} \, n^{2\lambda - 1} \qquad \text{as $n \to \infty$.}
\end{equation}
Furthermore, the relation
\begin{equation} \label{eq:gegenbauerC.orthogonality.relation}
  \int_{-1}^1\left(\GegenbauerC_n^{(\lambda)}(x)\right)^2(1-x^2)^{\lambda-\frac12}\dd x= 
  \frac{\sqrt{\pi} \, \gammafcn( \lambda + \frac{1}{2} )}{\gammafcn( \lambda +1 )} \, \frac{\lambda}{n + \lambda} \, \frac{\Pochhsymb{2\lambda}{n}}{n!}
\end{equation}
holds. We make frequent use of the Pochhammer symbol
\begin{equation*}
  (\alpha)_n=\alpha(\alpha+1)\cdots(\alpha+n-1)=
  \frac{\Gamma(n+\alpha)}{\Gamma(\alpha)}.
\end{equation*}
and the formulas
\begin{equation} \label{eq:pochhammer.rules}
\Pochhsymb{a}{2k} = 2^{2k} \Pochhsymb{\frac{a}{2}}{k} \Pochhsymb{\frac{a+1}{2}}{k}, \quad \Pochhsymb{a}{-k} = \frac{\gammafcn( a - k )}{\gammafcn( a )} = \frac{(-1)^k}{\Pochhsymb{1-a}{k}}, \quad \frac{\Pochhsymb{-n}{k}}{k!} = (-1)^k \binom{n}{k}.
\end{equation}
We also use the digamma function
\begin{equation*}
  \psi(x)=\frac{\Gamma'(x)}{\Gamma(x)}=
  -\gamma+\sum_{n=0}^\infty\left(\frac1{n+1}-\frac1{n+x}\right),
\end{equation*}
where $\gamma$ is the Euler-Mascheroni constant.

The classical hypergeometric functions are given by
\begin{equation*}
  \Hypergeom{p}{q}{a_1,\ldots,a_p}{b_1,\ldots,b_q}{z}=
  \sum_{n=0}^\infty\frac{(a_1)_n\cdots(a_p)_n}{(b_1)_n\cdots(b_q)_nn!}z^n
\end{equation*}
for $a_1,\ldots,a_p,b_1,\ldots,b_q\in\mathbb{C}$ and $p\leq q+1$. These power
series allow for an analytic continuation to the slit complex plane
$\mathbb{C}\setminus[1,\infty)$. For further properties of these functions we
refer to \cite{Andrews_Askey_Roy1999:special_functions,Luke-I:1969}.

We will state some of our results in terms of asymptotic series (see
\cite{Bruijn1958:asymptotic_methods_analysis}). We write
\begin{equation*}
  f(x)\sim\sum_{k=0}^\infty\phi_k(x)\quad\text{as }x\to\infty,
\end{equation*}
if for all $k\geq0$
\begin{equation*}
   \lim_{x\to\infty}\frac{\phi_{k+1}(x)}{\phi_k(x)}=0
\end{equation*}
and
\begin{equation*}
  f(x)-\sum_{\ell=0}^k\phi_\ell(x)=\mathcal{O}(\phi_{k+1}(x)).
\end{equation*}
In the statements of our results we will have sums of two and three asymptotic
series, which we understand in the following way
\begin{align*}
  f(x)&\sim\sum_{k=0}^\infty\phi_k(x)+\sum_{k=0}^\infty\psi_k(x)\\
  &\phantom{=}=  \phi_0(x)+\cdots+\phi_{k_1}(x)+\psi_0(x)+\cdots+\psi_{\ell_1}(x)+
  \phi_{k_1+1}(x)+\cdots\\
  &\phantom{=+}+\phi_{k_2}(x)+\psi_{\ell_1+1}(x)+\cdots,
\end{align*}
where
\begin{equation*}
  0=\lim_{x\to\infty}\frac{\phi_{k_1}(x)}{\psi_0(x)}=
  \lim_{x\to\infty}\frac{\psi_{\ell_1}(x)}{\phi_{k_1+1}(x)}=\cdots;
\end{equation*}
this means that we interlace the terms of the two series to obtain a new
asymptotic series. In the situation where we use this notation the two
constituting series will be so that this notation is well defined.

\section{Explicit formulas and generating functions}
\label{sec:expl-form-gener}

Let $\lambda >0$ and $\alpha, \beta > - 1$. We define for non-negative integers $n$, 
\begin{equation}\label{eq:In-def}
I_n^{(\lambda; \alpha, \beta)} \DEF \int_{-1}^1 \left( \GegenbauerC_n^{(\lambda)}( x ) \right)^2 \left( 1 - x \right)^\alpha \left( 1 + x \right)^\beta \dd x.
\end{equation}

\begin{thm} \label{thm:explicit}
Let $I_n^{(\lambda;\alpha,\beta)}$ be given by \eqref{eq:In-def}. Then we have
\begin{equation*}
\begin{split}
I_n^{(\lambda; \alpha, \beta)} 
&= 2^{\alpha+\beta+1} \frac{\gammafcn( \alpha + 1 ) \gammafcn( \beta + 1 )}{\gammafcn( \alpha + \beta + 2 )} \\
&\phantom{=}\times \left( \frac{\Pochhsymb{2\lambda}{n}}{n!} \right)^2 \Hypergeom{5}{4}{-n, n + 2\lambda, \lambda, \alpha + 1, \beta + 1}{2\lambda, \lambda + \frac{1}{2}, \frac{\alpha+\beta+2}{2}, \frac{\alpha+\beta+3}{2}}{1}.
\end{split}
\end{equation*}
\end{thm}

\begin{rmk} \label{rmk:explicit} For $\beta = \alpha=\mu-\frac12$, the
  $1$-balanced $\HyperpFq54$-hypergeometric polynomial reduces to a $1$-balanced
  $\HyperpFq43$-hypergeometric polynomial
\begin{equation*}
\Hypergeom{4}{3}{-n, n + 2\lambda, \lambda, \mu+\frac12}{2\lambda, \lambda + \frac{1}{2}, \mu +1}{1}.
\end{equation*}
To simplify notation, we set
\begin{equation*}
  J_n^{(\lambda;\mu)} \DEF I_n^{(\lambda;\mu-\frac12,\mu-\frac12)}, \qquad \lambda > 0, \mu > - \frac{1}{2}.
\end{equation*}

For $\mu = \lambda$, the $\HyperpFq43$ becomes
\begin{equation*}
\Hypergeom{3}{2}{-n, n + 2\lambda, \lambda}{2\lambda, \lambda + 1}{1},
\end{equation*}
hence can be computed by the Pfaff-Saalsch\"utz theorem as
\begin{equation*}
  \frac{\Pochhsymb{\lambda}{n} \Pochhsymb{-n}{n}}
  {\Pochhsymb{2\lambda}{n} \Pochhsymb{-n-\lambda}{n}}
  =\frac\lambda{n+\lambda}\frac{n!}{\Pochhsymb{2\lambda}{n}},
\end{equation*}
which, of course, reproduces the well known formula
\eqref{eq:gegenbauerC.orthogonality.relation} for the $L^2$-norm of
$\GegenbauerC_n^{(\lambda)}$ for weight $(1-x^2)^{\lambda-\frac12}$ (see
\cite{Andrews_Askey_Roy1999:special_functions}). 
\end{rmk}

\begin{proof}[Proof of Theorem~\ref{thm:explicit}]

The result follows from \cite[Eq.~(16)]{SanchezRuiz2001:linearization_and_connection_formula}, i.e.
\begin{equation*}
  \left( \GegenbauerC_n^{(\lambda)}( x ) \right)^2 =
  \left( \frac{\Pochhsymb{2\lambda}{n}}{n!} \right)^2
  \Hypergeom{3}{2}{-n, n + 2\lambda, \lambda}
  {2\lambda, \lambda + \frac{1}{2}}{1-x^2},
\end{equation*}
the relation 
\begin{equation*}
  \int_{-1}^1 (1-x^2)^k\left( 1 - x \right)^{\alpha}
  \left( 1 + x \right)^{\beta}\dd x
= 2^{2k + \alpha + \beta + 1} \frac{\gammafcn( k + \alpha + 1 )
  \gammafcn( k + \beta + 1 )}{\gammafcn( 2k + \alpha + \beta + 2 )},
\end{equation*}
and the duplication formula in \eqref{eq:pochhammer.rules} in order to rewrite the Pochhammer symbol $\Pochhsymb{\alpha+\beta+2}{2k}$.
\end{proof}

\begin{thm} \label{thm:generating.function.relation}
  The integrals $I_n^{(\lambda; \alpha, \beta)}$ satisfy the following generating
function relation
\begin{equation*}
\begin{split}
  \mathcal{I}^{(\lambda;\alpha,\beta)}(z)&\DEF\sum_{n=0}^\infty \frac{n!}{\Pochhsymb{2\lambda}{n}}  \,
  I_n^{(\lambda; \alpha, \beta)} \, z^n\\
&= 2^{\alpha+\beta+1} \frac{\gammafcn( \alpha + 1 ) \gammafcn( \beta + 1 )}{\gammafcn( \alpha + \beta + 2 )}
\frac{1}{( 1 - z )^{2\lambda}}
\Hypergeom{4}{3}{\lambda, \lambda, \alpha + 1, \beta + 1}
{2\lambda, \frac{\alpha+\beta+2}{2},
  \frac{\alpha+\beta+3}{2}}{-\frac{4z}{(1-z)^2}}.
\end{split}
\end{equation*}
\end{thm} 

\begin{rmk} \label{rmk:generating.function.relation}
  For $\alpha = \beta = \mu-\frac12$, we get with the help of
  \texttt{Mathematica}~12,
\begin{equation*}
  \mathcal{J}^{(\lambda;\mu)}(z)\DEF
  \sum_{n=0}^\infty \frac{n!}{\Pochhsymb{2\lambda}{n}} \,
  J_n^{(\lambda; \mu)} \, z^n
  = \frac{\sqrt{\pi}\gammafcn( \mu + \frac12 )}
  { \, \gammafcn( \mu + 1 )} \,
  \frac{1}{( 1 - z )^{2\lambda}}
  \Hypergeom{3}{2}{\lambda, \lambda, \mu + \frac12}
  {2\lambda, \mu + 1}
  {-\frac{4z}{(1-z)^2}}.
\end{equation*}
The same right-hand side is obtained for $\alpha = \mu + \frac12$ and
$\beta = \mu-\frac12$. This is obvious by the fact that
\begin{equation*}
  \int_{-1}^1\left(\GegenbauerC_n^{(\lambda)}(x)\right)^2(1-x)(1-x^2)^{\mu-\frac12}\dd x=
  \int_{-1}^1\left(\GegenbauerC_n^{(\lambda)}(x)\right)^2(1-x^2)^{\mu-\frac12}\dd x.
\end{equation*}
\end{rmk}

\begin{proof}[Proof of Theorem~\ref{thm:generating.function.relation}]
By Theorem~\ref{thm:explicit},
\begin{equation*}
\begin{split}
  &\sum_{n=0}^\infty \frac{n!}{\Pochhsymb{2\lambda}{n}}  \,
  I_n^{(\lambda; \alpha, \beta)} \, z^n 
  = \underbrace{2^{\alpha+\beta+1}
    \frac{\gammafcn( \alpha + 1 ) \gammafcn( \beta + 1 )}
    {\gammafcn( \alpha + \beta + 2 )}}_{A} \\
  &\phantom{equals}\times \sum_{n=0}^\infty \sum_{\ell=0}^n
  \frac{\Pochhsymb{2\lambda}{n+\ell}}{(n-\ell)!} \,
  \underbrace{\frac{\Pochhsymb{\lambda}{\ell} \Pochhsymb{\alpha+1}{\ell}
      \Pochhsymb{\beta+1}{\ell}}{\Pochhsymb{2\lambda}{\ell}
      \Pochhsymb{\lambda + \frac{1}{2}}{\ell}
      \Pochhsymb{\frac{\alpha+\beta+2}{2}}{\ell}
      \Pochhsymb{\frac{\alpha+\beta+3}{2}}{\ell}} \,
    \frac{(-1)^\ell}{\ell!}}_{c_\ell} \, z^n.
\end{split}
\end{equation*}
Interchanging order of summation,
\begin{equation*}
  \sum_{n=0}^\infty \frac{n!}{\Pochhsymb{2\lambda}{n}}  \,
  I_n^{(\lambda; \alpha, \beta)} \, z^n =
  A \sum_{\ell=0}^\infty \sum_{n=\ell}^\infty c_\ell \,
  \frac{\Pochhsymb{2\lambda}{n+\ell}}{(n-\ell)!} \,  z^n,
\end{equation*}
and taking into account that with the help of \texttt{Mathematica}~12,
\begin{equation*}
  \sum_{n=\ell}^\infty \frac{\Pochhsymb{2\lambda}{n+\ell}}{(n-\ell)!} \, z^n
  = \frac{1}{( 1 - z )^{2\lambda}} \, \Pochhsymb{\lambda}{\ell}
  \Pochhsymb{\lambda+\frac{1}{2}}{\ell} \left( \frac{4z}{(1-z)^2} \right)^\ell,
\end{equation*}
we arrive at the series expansion of the desired hypergeometric function.
\end{proof}

\section{Singularity analysis}\label{sec:singularity-analysis}
In the last section we have found generating functions for the quantities
$I_n^{(\lambda;\alpha,\beta)}$ and $J_n^{(\lambda;\mu)}$. In order to retrieve
asymptotic information about these quantities from analytic information about
the generating function at its singularity, we briefly discuss the method of
\emph{singularity analysis} introduced in
\cite{Flajolet_Odlyzko1990:singularity_analysis_generating}. The main advantage
of this method over the classical method of Darboux (see
\cite{Darboux1878:memoire_sur_lapproximation_1,
  Darboux1878:memoire_sur_lapproximation_2}) is that this method is also able
to obtain asymptotic expressions for the coefficents of generating functions in
the case that the coefficients tend to $0$. This difference comes from the fact
that Darboux's method uses a local approximation of the generating function
\emph{inside} the circle of convergence and uses the Riemann-Lebesgue-lemma to
obtain an error term. Singularity analysis needs information on the behaviour
of the analytic continuation to a region of the form
\begin{equation*}
  \Delta_{\varepsilon,\phi}=
  \left\{z\in\mathbb{C} \  \big| \  |z|<1+\varepsilon, |\arg(1-z)|<\phi \right\}
\end{equation*}
for some $\pi>\phi>\frac\pi2$ (assuming that the radius of convergence is $1$).
Since in our case the generating functions have an analytic continuation
the complex plane with a branch cut connecting $1$ and $\infty$, the method is
readily applicable.

The main ingredient of the method is the following theorem.
\begin{thm}[Big-$\mathcal{O}$-theorem, see
  {\cite[Theorem~1]{Flajolet_Odlyzko1990:singularity_analysis_generating}}]
  \label{thm:big-o}
  Assume that, with the sole exception of the singularity $z=1$, $f(z)$ is
  analytic in $\Delta_{\varepsilon,\phi}$ for some $\varepsilon>0$ and
  $\phi>\frac\pi2$. Assume further that as $z$ tends to $1$ in
  $\Delta_{\varepsilon,\phi}$,
  \begin{equation*}
    f(z)=\mathcal{O}(|1-z|^\alpha)
  \end{equation*}
  for some real number $\alpha$. Then the $n$-th Taylor coefficient of $f(z)$
  satisfies
  \begin{equation*}
    f_n=[z^n]f(z)=\mathcal{O}(n^{-\alpha-1}).
  \end{equation*}
\end{thm}
As a consequence of this theorem a local expansion around $z=1$ of the
generating function
\begin{equation*}
  f(z)=\sum_{k=0}^K a_k(1-z)^{\alpha_k}+\mathcal{O}\left(|1-z|^\beta\right)
\end{equation*}
for $\alpha_0<\alpha_1<\cdots<\alpha_K<\beta$ translates into an asymptotic
relation for the coefficients
\begin{equation*}
  f_n=\sum_{k=0}^K a_k\binom{n-\alpha_k-1}n+\mathcal{O}(n^{-\beta-1}).
\end{equation*}
Each of the binomial coefficients has an asymptotic expansion in terms of
powers of $n$:
\begin{equation} \label{eq:sing.analysis.binomial.asymptotics}
\binom{n-\alpha_k-1}{n} = \frac{1}{\gammafcn( - \alpha_k )} \, \frac{\gammafcn( n - \alpha_k )}{\gammafcn( n + 1 )} = \frac{n^{-\alpha_k-1}}{\gammafcn( - \alpha_k )} \left\{ 1 + \mathcal{O}\Big( \frac{1}{n} \Big) \right\} \qquad \text{as $n \to \infty$.}
\end{equation}
%

The paper \cite{Flajolet_Odlyzko1990:singularity_analysis_generating} also
contains more general theorems of this type suitable for more complicated
asymptotic behaviour of $f(z)$ for $z\to1$, like logarithmic singularities.

\section{Mellin-Barnes formulas}\label{sec:mell-barn-form}
In order to write the generating function
$\mathcal{I}^{(\lambda;\alpha,\beta)}(z)$ in a more tractable form, we recall
the Mellin-Barnes formula for hypergeometric functions (see
\cite{Andrews_Askey_Roy1999:special_functions,
  Paris_Kaminski2001:asymptotics_mellin_barnes}). This gives
\begin{equation}
  \label{eq:I-mellin-barnes}
  \mathcal{I}^{(\lambda;\alpha,\beta)}(z)=
  \frac{2^{\alpha+\beta+1}\Gamma(2\lambda)}{\Gamma(\lambda)^2(1-z)^{2\lambda}}
  \frac1{2\pi i}\int\limits_{-i\infty}^{i\infty}
  \frac{\Gamma(s+\lambda)^2\Gamma(s+\alpha+1)\Gamma(s+\beta+1)\Gamma(-s)}
  {\Gamma(s+2\lambda)\Gamma(2s+\alpha+\beta+2)}
  \left(\frac{16z}{(1-z)^2}\right)^s \dd s,
\end{equation}
where the contour of integration is taken along the imaginary axis encircling
$s=0$ in the left half-plane such that the poles at $s=-\lambda$, $s=-\alpha-1$,
and $s=-\beta-1$ are to the left of the contour. Similarly, we obtain
\begin{equation}
  \label{eq:J-mellin-barnes}
  \mathcal{J}^{(\lambda;\mu)}(z)=
  \frac{\sqrt\pi \, \Gamma(2\lambda)}{\Gamma(\lambda)^2(1-z)^{2\lambda}}
  \frac1{2\pi i}\int\limits_{-i\infty}^{i\infty}
  \frac{\Gamma(s+\lambda)^2\Gamma(s+\mu+\frac12)\Gamma(-s)}
  {\Gamma(s+2\lambda)\Gamma(s+\mu+1)}\left(\frac{4z}{(1-z)^2}\right)^s 
    \dd s,
  \end{equation}
where the contour is chosen as before, this time leaving $s=-\lambda$ and
$s=-\mu-\frac12$  to the left of the contour.

\section{Jacobi weights, generic case}\label{sec:jacobi-weights}
We use the formula \eqref{eq:I-mellin-barnes} to derive an asymptotic expansion
of $\mathcal{I}^{(\lambda;\alpha,\beta)}(z)$ around $z=1$. This expansion is
then translated into a full asymptotic expansion of
$I_n^{(\lambda;\alpha,\beta)}$ in the generic case. In this case the integrand
in~\eqref{eq:I-mellin-barnes} has simple poles at $-\alpha-1-\ell$,
$-\beta-1-\ell$ and double poles at $-\lambda-\ell$ for
$\ell \in \mathbb{N}_0$. There is no pole cancellation or pole multiplication;
i.e., $\alpha-\lambda$, $\beta-\lambda$, $\alpha-\beta$, $\alpha-2\lambda$,
$\beta-2\lambda$, $\alpha+\beta-2\lambda$, and $\lambda$ are no integers.

Moving the contour in \eqref{eq:I-mellin-barnes} to the left and collecting the
residues at the double poles at $-\lambda-\ell$, we have
\begin{equation*}
\begin{split}
&\frac{\gammafcn( \lambda + \frac{1}{2} ) \gammafcn( 1 + \alpha - \lambda )
  \gammafcn( 1 + \beta - \lambda ) }{2^{2\lambda-1-\alpha-\beta} \sqrt{\pi} \,
  \gammafcn( \lambda ) \gammafcn( 2 + \alpha + \beta - 2\lambda )}\\
&\times z^{-\lambda} \Hypergeom{4}{3}{1-\lambda, \lambda, \frac{2\lambda-\alpha-\beta-1}{2}, \frac{2\lambda-\alpha-\beta}{2}}{1, \lambda - \alpha, \lambda - \beta}{- \frac{(1-z)^2}{4z}} \, \log\frac{1}{1-z}
+ \text{power series in $(1-z)$;}
\end{split}
\end{equation*}
collecting the residues at the simple poles at $-\alpha-1-\ell$, we get
\begin{equation*}
\begin{split}
&\frac{\gammafcn( \lambda + \frac{1}{2} ) \gammafcn( \alpha +1 ) \gammafcn( \lambda - \alpha - 1 )^2 }{2^{3\alpha + 4 - \beta - 2\lambda} \sqrt{\pi} \, \gammafcn( \lambda ) \gammafcn( 2\lambda - \alpha - 1 )} \\
&\times\left( 1 - z \right)^{2 + 2 \alpha - 2 \lambda} z^{-1-\alpha} \Hypergeom{4}{3}{\alpha + 1, \alpha + 2 - 2\lambda, \frac{1+\alpha-\beta}{2}, \frac{2+\alpha-\beta}{2}}{1 + \alpha - \beta, 2 + \alpha - \lambda, 2 + \alpha - \lambda}{- \frac{(1-z)^2}{4z}};
\end{split}
\end{equation*}
and collecting the residues at the simple poles at $-\beta-1-\ell$, we obtain
\begin{equation*}
\begin{split}
&\frac{\gammafcn( \lambda + \frac{1}{2} ) \gammafcn( \beta +1 ) \gammafcn( \lambda - \beta - 1 )^2 }{2^{3\beta + 4 - \alpha - 2\lambda} \sqrt{\pi} \, \gammafcn( \lambda ) \gammafcn( 2\lambda - \beta - 1 )} \\
&\times\left( 1 - z \right)^{2 + 2\beta -2 \lambda} z^{-1-\beta} \Hypergeom{4}{3}{\beta + 1, \beta + 2 - 2\lambda, \frac{1+\beta-\alpha}{2}, \frac{2+\beta-\alpha}{2}}{1 + \beta - \alpha, 2 + \beta - \lambda, 2 + \beta - \lambda}{- \frac{(1-z)^2}{4z}};
\end{split}
\end{equation*}
Thus, we arrive at
\begin{equation} \label{eq:I-asymp-1}
\begin{split}
\mathcal{I}^{(\lambda;\alpha,\beta)}(z) 
&= \sum_{m = 0}^\infty A_m \left( 1 - z \right)^{m + 2 + 2\alpha - 2\lambda} + \sum_{m = 0}^\infty B_m \left( 1 - z \right)^{m + 2 + 2\beta -2 \lambda} \\
&\phantom{=}+ \sum_{m = 0}^\infty D_m \left( 1 - z \right)^m \log \frac{1}{1-z} + \text{power series in $(1-z)$},
\end{split}
\end{equation}
where 
\begin{align*}
A_m 
&\DEF \frac{\gammafcn( \lambda + \frac{1}{2} ) \gammafcn( \alpha +1 ) \gammafcn( \lambda - \alpha - 1 )^2 }{2^{3\alpha + 4 - \beta - 2\lambda} \sqrt{\pi} \, \gammafcn( \lambda ) \gammafcn( 2\lambda - \alpha - 1 )} \sum_{\ell=0}^{\lfloor \frac{m}{2} \rfloor} \frac{\Pochhsymb{\alpha+2-2\lambda}{\ell} \Pochhsymb{1+\alpha-\beta}{2\ell} \Pochhsymb{\alpha+1}{m-\ell}}{\Pochhsymb{1+\alpha-\beta}{\ell} \Pochhsymb{2+\alpha-\lambda}{\ell}^2 \ell! (m - 2\ell)!} \, \frac{(-1)^\ell}{16^\ell}, \\
B_m 
&\DEF \frac{\gammafcn( \lambda + \frac{1}{2} ) \gammafcn( \beta +1 ) \gammafcn( \lambda - \beta - 1 )^2 }{2^{3\beta + 4 - \alpha - 2\lambda} \sqrt{\pi} \, \gammafcn( \lambda ) \gammafcn( 2\lambda - \beta - 1 )} \sum_{\ell=0}^{\lfloor \frac{m}{2} \rfloor} \frac{\Pochhsymb{\beta+2-2\lambda}{\ell} \Pochhsymb{1+\beta-\alpha}{2\ell} \Pochhsymb{\beta+1}{m-\ell}}{\Pochhsymb{1+\beta-\alpha}{\ell} \Pochhsymb{2+\beta-\lambda}{\ell}^2 \ell! (m - 2\ell)!} \, \frac{(-1)^\ell}{16^\ell}, \\
D_m
&\DEF \frac{\gammafcn( \lambda + \frac{1}{2} ) \gammafcn( \alpha + 1 - \lambda ) \gammafcn( \beta + 1 - \lambda ) }{2^{2\lambda-\alpha-\beta-1} \sqrt{\pi} \, \gammafcn( \lambda ) \gammafcn( \alpha + \beta + 2 - 2\lambda )} \sum_{\ell=0}^{\lfloor \frac{m}{2} \rfloor} \frac{\Pochhsymb{1-\lambda}{\ell} \Pochhsymb{2\lambda - \alpha - \beta - 1}{2\ell} \Pochhsymb{\lambda}{m-\ell}}{\Pochhsymb{\lambda-\alpha}{\ell} \Pochhsymb{\lambda-\beta}{\ell} \ell! \ell! (m - 2\ell)!} \, \frac{(-1)^\ell}{16^\ell}.
\end{align*}

The relation \eqref{eq:I-asymp-1} holds as an asymptotic relation at
first. Since $\mathcal{I}^{(\lambda;\alpha,\beta)}(z)$ has an analytic
continuation to $\mathbb{C}\setminus[1,\infty)$ and  satisfies a fourth order
differential equation with regular singular points $0,1,\infty$ (as a
consequence of the representation in terms of hypergeometric functions), it has
a power series representation of the form \eqref{eq:I-asymp-1} with radius of
convergence $\geq1$ by the Frobenius method. The asymptotic expansion has to
coincide with this power series representation.

The terms \emph{power series in} $(1-z)$ correspond to a function holomorphic
around $z=1$. Since this function does not contribute to the asymptotic
expansion we are aiming for, we do not work out these terms, which are slightly
more elaborate than the remaining terms.

By Theorem~\ref{thm:big-o} the local expansion \eqref{eq:I-asymp-1} around
$z=1$ translates into an asymptotic series for the coefficients
\begin{equation}\label{eq:In-asymp}
  \begin{split}
    I_n^{(\lambda;\alpha,\beta)}&\sim
    \frac{(2\lambda)_n}{n!}\Biggl( \sum_{m=0}^\infty
    D_m(-1)^m\frac{m!}{n(n-1)\cdots(n-m)}\\
    &\phantom{equ}+\sum_{m=0}^\infty
      A_m\binom{n+2\lambda-2\alpha-3-m}{n}+ \sum_{m=0}^\infty
      B_m\binom{n+2\lambda-2\beta-3-m}{n}\Biggr).
  \end{split}
\end{equation}

In order to make the results more transparent and applicable, we state the
asymptotic main terms as a theorem.
\begin{thm}\label{thm:jacobi}
  Let $-1<\alpha<\beta$ and $\lambda>0$ be real numbers. Then
  \begin{equation}
    \label{eq:jacobi-asymp-1}
    I_n^{(\lambda;\alpha,\beta)}=
    \begin{cases}
      \dfrac{2^{\alpha+\beta+2-4\lambda} \Gamma(\alpha+1-\lambda)
        \Gamma(\beta+1-\lambda)}{\Gamma(\lambda)^2
        \Gamma(\alpha+\beta+2-2\lambda)} \, n^{2\lambda-2}+\mathcal{O}(n^\eta)
      & \alpha>\lambda-1,\\
      \dfrac{2^{\beta+2-3\lambda}}{\Gamma(\lambda)^2} \, n^{2\lambda-2} 
      \Big(\log n - A(\lambda,\beta) \Big) + \mathcal{O}\Big( n^{2\lambda - 3} \log n \Big) + \mathcal{O}\Big( n^{4\lambda - 2\beta - 5} \Big) & \alpha=\lambda-1,\\
      \dfrac{2^{\beta -  3\alpha - 3} \gammafcn( \alpha +1 ) \gammafcn( \lambda - \alpha - 1 )^2 }{\gammafcn( \lambda )^2 \gammafcn( 2\lambda - \alpha - 1 )} \, 
      n^{4\lambda-2\alpha-4}
      +\mathcal{O}(n^\eta) & \alpha<\lambda-1,
    \end{cases}
  \end{equation}
  where
  \begin{equation*}
    \eta=
    \begin{cases}
      \max(2\lambda-3,4\lambda-2\alpha-4) & \alpha>\lambda-1,\\
      \max(2\lambda-2,4\lambda-2\alpha-5,4\lambda-2\beta-4)& \alpha<\lambda-1,
    \end{cases}
  \end{equation*}
and 
\begin{equation*}
A(\lambda,\beta) = \frac{1}{2} \left( \gamma - 4 \log 2 + \digammafcn( \beta + 1 - \lambda ) + 2 \digammafcn( \lambda ) \right).
\end{equation*}

\end{thm}
\begin{proof}
  The asymptotic main term in \eqref{eq:In-asymp} is given by the first term of
  the first series, if $\alpha>\lambda-1$; it is given by the first term of the
  second series, if $\alpha<\lambda-1$. Except when $\alpha = \lambda - 1$, this also holds in the non-generic case, since only higher order asymptotic terms would be affected. 
%
%

It remains to discuss the case $\alpha=\lambda-1$. 
In this case, when taking into account a triple pole of the integrand in \eqref{eq:I-mellin-barnes} at $-\lambda$, the generating function has the local expansion
\begin{equation*}
\begin{split}
\mathcal{I}^{(\lambda;\lambda-1,\beta)}
&= 2^{\beta-\lambda} \frac{\gammafcn( \lambda + \frac{1}{2} )}{\sqrt{\pi} \, \gammafcn( \lambda )} \Bigg( \left( \log \frac{1}{1-z} \right)^2 - \Big( 3 \gamma - 4 \log 2  + \digammafcn( 1 + \beta - \lambda ) + 2 \digammafcn( \lambda ) \Big) \log \frac{1}{1-z} \Bigg) \\
&\phantom{=}+ C + \mathcal{O}\Big( (1 - z) ( \log( 1 - z ) )^2 \Big) + \mathcal{O}\Big( \left( 1 - z \right)^{2( \beta + 1 - \lambda )} \Big),
\end{split}
\end{equation*}
where $C$ is a constant that will play no role in the singularity analysis.  
This relation translates into the asymptotic expansion
\begin{equation*}
\begin{split}
I_n^{(\lambda;\lambda-1,\beta)}
&= \frac{\Pochhsymb{2\lambda}{n}}{n!} \, 2^{\beta-\lambda} \frac{\gammafcn( \lambda + \frac{1}{2} )}{\sqrt{\pi} \, \gammafcn( \lambda )} \, \Bigg( \frac{2}{n} \sum_{k=1}^{n-1} \frac{1}{k} - \Big( 3 \gamma - 4 \log 2  + \digammafcn( 1 + \beta - \lambda ) + 2 \digammafcn( \lambda ) \Big) \, \frac{1}{n} \\
&\phantom{=+}+ \mathcal{O}\Big( \frac{\log n}{n^2} \Big) + \mathcal{O}\Big( \frac{1}{n^{2( \beta + 1 - \lambda )+1}} \Big) \Bigg).
\end{split}
\end{equation*}
%
%
%
Using the asymptotic expansion
\begin{equation*}
  \sum_{k=1}^{n-1} \frac1k = \digammafcn( n ) + \gamma = \log n + \gamma +\mathcal{O}\Big(\frac1n\Big),
\end{equation*}
we obtain the stated expression. 

We remark that the leading asymptotic term in the case $\alpha = \lambda - 1$ could be obtained by taking the limit as $a \to \lambda-1$ in 
\begin{equation*}
\frac{\Pochhsymb{2\lambda}{n}}{n!} \left( \frac{D_0}{n} + A_0 \binom{n+2\lambda-2\alpha-3}{n} \right)
\end{equation*}
derived from the general asymptotics~\eqref{eq:In-asymp}.
\end{proof}

\section{Gegenbauer weights, generic case}\label{sec:gegenbauer-weights}
The Gegenbauer weights are a special case of Jacobi weights in the non-generic setting. We 
present the asymptotic evaluation of the integrals
%
\begin{equation}\label{eq:Jn-def}
  J_n^{(\lambda;\mu)} \DEF \int_{-1}^1\left( \GegenbauerC_n^{(\lambda)}(x)\right)^2 (1-x^2)^{\mu-\frac12}\dd x  = I_n^{(\lambda;\mu-\frac12,\mu-\frac12)}
\end{equation}
using the more appropriate specialisation $\alpha = \beta = \mu - \frac{1}{2}$, so that for $\mu = \lambda$ we get back the well-known result~\eqref{eq:gegenbauerC.orthogonality.relation}.  
%
%

Because of the symmetry in the Gegenbauer weight, there is a second obvious approach to a formula for $J_n^{(\lambda;\mu)}$ based on
connection formulas between $\GegenbauerC_n^{(\lambda)}$ and $\GegenbauerC_n^{(\mu)}$.
\begin{thm}\label{thm:gegen} Let $\mu > -\frac{1}{2}$ and $\lambda > 0$. Let $J_n^{(\lambda;\mu)}$ be given by
  \eqref{eq:Jn-def}. Then we have
  \begin{equation*}
    J_n^{(\lambda;\mu)}=\frac{\sqrt\pi\,\Gamma(\mu+\frac12)}{\Gamma(\mu+1)}
    \left(\frac{(2\lambda)_n}{n!}\right)^2
    \Hypergeom{4}{3}{-n,n+2\lambda,\lambda,\mu+\frac12}
    {2\lambda,\lambda+\frac12,\mu+1}{1};
  \end{equation*}
  alternatively, we have
  \begin{equation}\label{eq:Jn-connect}
    J_n^{(\lambda;\mu)}=\frac{\sqrt\pi\,\Gamma(\mu+\frac12)}{\mu \, \Gamma(\mu+1)}
    \sum_{k=0}^{\lfloor\frac n2\rfloor} 
    \frac{(\lambda)_{n-k}^2(\lambda-\mu)_k^2}{(\mu+1)_{n-k}^2(k!)^2}
    (n+\mu-2k) \frac{(2\mu)_{n-2k}}
    {(n-2k)!}.
  \end{equation}
\end{thm}
\begin{proof}
  The first equation is an immediate consequence of Theorem~\ref{thm:explicit}
  after specialising $\alpha=\beta=\mu-\frac12$. The second expression can be
  obtained using the connection formula
  \begin{equation*}
    \GegenbauerC_n^{(\lambda)}(x)=\sum_{k=0}^{\lfloor\frac n2\rfloor}
    \frac{(\lambda)_{n-k}(\lambda-\mu)_k}{(\mu+1)_{n-k}k!}\frac{n+\mu-2k}\mu
    \GegenbauerC_{n-2k}^{(\mu)}(x)
  \end{equation*}
  (see \cite[(7.1.11)]{Andrews_Askey_Roy1999:special_functions}) and relation~\eqref{eq:gegenbauerC.orthogonality.relation}. 
\end{proof}
\begin{rmk}
  The two formulas for $J_n^{(\lambda;\mu)}$ give the identity
  \begin{equation*}
   \sum_{k=0}^{\lfloor\frac n2\rfloor}
    \frac{(\lambda)_{n-k}^2(\lambda-\mu)_k^2(2\mu)_{n-2k}(n+\mu-2k)}
    {(\mu+1)_{n-k}^2(k!)^2(n-2k)!}=\mu  \left(\frac{(2\lambda)_n}{n!}\right)^2
    \Hypergeom{4}{3}{-n,n+2\lambda,\lambda,\mu+\frac12}
    {2\lambda,\lambda+\frac12,\mu+1}{1}.
  \end{equation*}
  Notice that the sum on the left-hand side has only positive terms, whereas
  the (implicit) sum on the right-hand side is alternating. Alternatively, this
  identity could be proved using Zeilberger's algorithm (see
  \cite{Kauers_Paule2011:concrete_tetrahedron}). With the help of
  \texttt{Mathematica} and using the implementation
  \cite{RISC:computer_algebra_combinatorics} of Zeilberger's algorithm we found
  that both expressions satisfy the linear recurrence relation
  \begin{equation*}
    (n+2\lambda)^2 (n+2\lambda-\mu)J_n^{(\lambda;\mu)} -
    2(n+\lambda+1)(n^2+2(\lambda+1)n+ 3 \lambda+1)J_{n+1}^{(\lambda;\mu)}
    +(n+2)^2 (n + \mu + 2) J_{n+2}^{(\lambda;\mu)}=0.
  \end{equation*}
This relation could be used to give an independent proof of above identity.
\end{rmk}
The generating function
\begin{equation*}
  \mathcal{J}^{(\lambda;\mu)}(z)=\sum_{n=0}^\infty
  \frac{n!}{(2\lambda)_n}J_n^{(\lambda;\mu)}z^n
\end{equation*}
satisfies the differential equation
\begin{multline*}
  z^2(z-1)^2y^{\prime\prime\prime} + \left( z - 1 \right) z \left( (4 ( \lambda + 1 ) -\mu ) z - 2 ( \lambda + 1 ) -\mu \right) y^{\prime\prime} \\
  +2\left( \left( \left( \lambda + 1 \right) \left( 2( \lambda + 1 ) - \mu \right) - 1 \right) z^2 -
     \left(2 \left( \lambda + 1 \right)^2 - 1 \right) z+ \lambda  (\mu +1) \right) y'\\
  +2 \lambda  ((2 \lambda-\mu)  z-\lambda)y=0.
\end{multline*}
This is a third order differential equation with regular singular points at
$0,1,\infty$. The indicial equation at $z=1$ reads as
\begin{equation*}
  x^2 (x+ 2 \lambda - 2 \mu-1)=0.
\end{equation*}
Thus, the fundamental solutions about $1$ will take the form of a power series
in $(1-z)$, $\log \frac{1}{1-z}$ times a power series in $(1-z)$, and
$(1-z)^{1+2\mu-2\lambda}$ times a power series in $(1-z)$.

From Remark~\ref{rmk:generating.function.relation} we have that
\begin{equation}\label{eq:Jz}
  \mathcal{J}^{(\lambda;\mu)}(z)=
  \frac{\sqrt\pi \, \Gamma(\mu+\frac12)}{\Gamma(\mu+1)}
  \frac1{(1-z)^{2\lambda}}\Hypergeom{3}{2}{\lambda,\lambda,\mu+\frac12}
  {2\lambda,\mu+1}{-\frac{4z}{(1-z)^2}}.
\end{equation}
The asymptotic behaviour of $J_n^{(\lambda;\mu)}$ is encoded in the behaviour
of the function $\mathcal{J}^{(\lambda;\mu)}(z)$ around $z=1$ (see
\cite{Flajolet_Odlyzko1990:singularity_analysis_generating}). 
 To obtain this local expansion we proceed as before by using
 \eqref{eq:J-mellin-barnes} and shifting the line of integration to the left. 

We consider the generic case for $\lambda \notin \mathbb{Z}$, $\mu + 1 - \lambda \notin \mathbb{Z}$, $\mu + \frac{1}{2} - 2\lambda \notin \mathbb{Z}$, and $\mu + \frac{1}{2} - \lambda \notin \mathbb{Z}$. 

 Collecting residues at the double poles $-\lambda-\ell$, $\ell \in \mathbb{N}_0$, we get 
 \begin{align*}
 \frac{\gammafcn( \lambda + \frac{1}{2} ) \gammafcn( \frac{1}{2} + \mu - \lambda )}{\gammafcn( \lambda ) \gammafcn( 1 + \mu - \lambda )} 
 &\Bigg( z^{-\lambda} \Hypergeom{3}{2}{1-\lambda, \lambda, \lambda - \mu}{1, \frac{1}{2} + \lambda - \mu}{- \frac{( 1 - z )^2}{4z}} \, \log \frac{1}{1-z} \\
 &\phantom{\Bigg( }+ \text{power series in $(1-z)$} \Bigg).
 \end{align*}
Collecting residues at the simple poles $-\mu-\frac{1}{2}-\ell$, $\ell \in \mathbb{N}_0$, we obtain
\begin{equation*}
\frac{\gammafcn( \lambda + \frac{1}{2} ) \gammafcn( \lambda - \mu - \frac{1}{2} )^2 \gammafcn( \mu + \frac{1}{2} ) }{4^{1+\mu-\lambda} \sqrt{\pi} \, \gammafcn( \lambda ) \gammafcn( 2 \lambda - \mu - \frac{1}{2} )} \, \frac{\left( 1 - z \right)^{1+2\mu-2\lambda}}{ z^{\mu+\frac{1}{2}}} \Hypergeom{3}{2}{\frac{1}{2}, \mu + \frac{1}{2}, \frac{3}{2} + \mu - 2 \lambda}{\frac{3}{2} + \mu - \lambda, \frac{3}{2} + \mu - \lambda}{- \frac{( 1 - z )^2}{4z}}.
\end{equation*}
Thus, we arrive at
\begin{equation*}
\mathcal{J}^{(\lambda;\mu)}(z) = \sum_{m=0}^\infty A_m \left( 1 - z \right)^m \log \frac{1}{1-z} + \sum_{m=0}^\infty B_m \left( 1 - z \right)^{1+2\mu-2\lambda+m} + \text{power series in $(1-z)$},
\end{equation*}
where 
\begin{align*}
A_m &= \frac{\gammafcn( \lambda + \frac{1}{2} ) \gammafcn( \frac{1}{2} + \mu - \lambda )}{\gammafcn( \lambda ) \gammafcn( 1 + \mu - \lambda )}  \sum_{\ell=0}^{\lfloor \frac{m}{2} \rfloor} \frac{\Pochhsymb{1-\lambda}{\ell} \Pochhsymb{\lambda-\mu}{\ell} \Pochhsymb{\lambda}{m-\ell}}{ \Pochhsymb{\frac{1}{2}+\lambda-\mu}{\ell} \ell! \ell! (m-2\ell)!} \, \frac{(-1)^\ell}{4^\ell}, \\
B_m &= \frac{\gammafcn( \lambda + \frac{1}{2} ) \gammafcn( \lambda - \mu - \frac{1}{2} )^2 \gammafcn( \mu + \frac{1}{2} ) }{4^{1+\mu-\lambda} \sqrt{\pi} \, \gammafcn( \lambda ) \gammafcn( 2 \lambda - \mu - \frac{1}{2} )}\sum_{\ell=0}^{\lfloor \frac{m}{2} \rfloor} \frac{\Pochhsymb{\frac{1}{2}}{\ell} \Pochhsymb{\frac{3}{2} + \mu - 2 \lambda}{\ell} \Pochhsymb{\mu + \frac{1}{2}}{m-\ell}}{\Pochhsymb{\frac{3}{2}+\mu-\lambda}{\ell} \Pochhsymb{\frac{3}{2}+\mu-\lambda}{\ell} \ell! (m - 2\ell)!} \, \frac{(-1)^\ell}{4^\ell}.
\end{align*}

Using singularity analysis, this translates into the following asymptotic series:
\begin{equation} \label{eq:Jn.lambda.mu.asymptotics.generic}
J_n^{(\lambda;\mu)} = \frac{(2\lambda)_n}{n!} \left( \sum_{m=0}^\infty (-1)^m \frac{m!}{n(n-1)\cdots(n-m)} A_m + \sum_{m=0}^\infty \binom{n+2\lambda-2\mu-2-m}{n} B_m  \right).
\end{equation}
%
%


Regarding the leading term, we have the following result. 
\begin{thm}
  Let $\mu > -\frac12$ and $\lambda>0$. Then 
  \begin{equation}
    \label{eq:jacobi-asymp}
    J_n^{(\lambda;\mu)}=
    \begin{cases}
      \dfrac{\sqrt\pi \, \Gamma(\mu+\frac12-\lambda)}{2^{2\lambda-1} \Gamma(\lambda)^2
      \Gamma(\mu+1-\lambda)} \, n^{2\lambda-2} + \mathcal{O}(n^\eta)
     &\mu>\lambda-\frac12,\\
      \dfrac{1}{2^{2\lambda-2} \Gamma(\lambda)^2} \, n^{2\lambda-2} \left(\log n + 2 \log 2 - \digammafcn( \lambda ) \right) + \mathcal{O}\left(n^{2\lambda-3} \log n \right)
      &\mu=\lambda-\frac12,\\
      \dfrac{\sqrt{\pi} \, \Gamma(\lambda-\mu-\frac12) \Gamma(\mu+\frac12)}
      {2^{2\lambda-1} \Gamma(\lambda)^2 \gammafcn( \lambda - \mu ) \Gamma(2\lambda-\mu-\frac12)}
      n^{4\lambda-2\mu-3}+\mathcal{O}(n^\eta)&\mu<\lambda-\frac12,
    \end{cases}
  \end{equation}
  where
  \begin{equation*}
    \eta=
    \begin{cases}
      \max(2\lambda-3,4\lambda-2\mu-3)&\mu>\lambda-\frac12,\\
      \max(2\lambda-2,4\lambda-2\mu-4)&\mu<\lambda-\frac12.
    \end{cases}
  \end{equation*}
\end{thm}
\begin{proof}
In the generic case, the asymptotic terms can be read off from the general asymptotics \eqref{eq:Jn.lambda.mu.asymptotics.generic} (setting $m=0$) and using \eqref{eq:Cnlambda.of.one} and \eqref{eq:sing.analysis.binomial.asymptotics}. The leading term (and the remainder term) is still valid in the non-generic case, except for $\mu = \lambda - \frac{1}{2}$, when the first two terms need to be combined to allow for a cancellation of the poles of $\Gamma(\mu+\frac12-\lambda)$ and $\Gamma(\lambda-\mu-\frac12)$ as $\mu \to \lambda - \frac{1}{2}$.  

It only remains to study the case $\mu=\lambda-\frac12$.
  Then the generating function becomes
\begin{equation*}
  \mathcal{J}^{(\lambda;\lambda-\frac12)}(z)
  =\frac{\sqrt\pi \,\Gamma(\lambda)}{\Gamma(\lambda+\frac12)}
  \frac1{(1-z)^{2\lambda}}\Hypergeom{3}{2}{\lambda,\lambda,\lambda}
  {2\lambda,\lambda+\frac12}{-\frac{4z}{(1-z)^2}}.
\end{equation*}
In this case the Mellin-Barnes formula reads as
\begin{equation}
  \label{eq:lambda-1/2}
  \mathcal{J}^{(\lambda;\lambda-\frac12)}(z)=
  \frac{\sqrt\pi \, \Gamma(2\lambda)}{\Gamma(\lambda)^2(1-z)^{2\lambda}}
  \frac1{2\pi i}\int\limits_{-i\infty}^{i\infty}
  \frac{\Gamma(s+\lambda)^3\Gamma(-s)}
  {\Gamma(s+2\lambda)\Gamma(s+\lambda+\frac12)}\left(\frac{4z}{(1-z)^2}\right)^s
  \dd s;
\end{equation}
where the integrand has triple poles at $s+\lambda\in-\mathbb{N}_0$.
Shifting the line of integration to $\Re(s)=-\lambda-1$ gives
\begin{equation}\label{eq:special-residue}
\begin{split}
  \mathcal{J}^{(\lambda;\lambda-\frac12)}(z)&=
  \frac{\Gamma(\lambda + \frac{1}{2})}{\sqrt\pi \, \Gamma(\lambda)} \, z^{-\lambda}
  \left(\gamma - \log 2 + \psi(\lambda) - \frac{1}{2} \log\frac{4z}{(1-z)^2} \right)^2\\
&+
  \frac{\sqrt\pi\,\Gamma(2\lambda)}{\Gamma(\lambda)^2(1-z)^{2\lambda}}
  \frac1{2\pi i}\int\limits_{-\lambda-1-i\infty}^{-\lambda-1+i\infty}
  \frac{\Gamma(s+\lambda)^3\Gamma(-s)}
  {\Gamma(s+2\lambda)\Gamma(s+\lambda+\frac12)}\left(\frac{4z}{(1-z)^2}\right)^s
  \dd s.
\end{split}
\end{equation}
The contour of integration is chosen along the vertical line
$\Re(s)=-\lambda-1$ with a small circular arc encircling $-\lambda-1$ to the
right.

The integral is $\mathcal{O}((1-z)^2(\log(1-z))^2)$, thus the first term is the
asymptotic main term. We note that in principle a full asymptotic expansion of
$\mathcal{J}^{(\lambda;\lambda-\frac12)}(z)$ around $z=1$ could be developed by
shifting the line of integration further to the left. The terms originating
from the residues of the triple poles at $s=-\lambda-\ell$, $\ell \in \mathbb{N}_0$, become more complicated. Thus we confine ourselves to
the first term.

This gives
\begin{equation}
  \label{eq:Jspecial-asymp}
  \begin{split}
    \mathcal{J}^{(\lambda;\lambda-\frac12)}(z)&=
    \frac{\Gamma(\lambda+\frac{1}{2})}{\sqrt\pi \, \Gamma(\lambda)}
    \left(\left(\log\frac1{1-z}\right)^2 - 
      2 \left( \gamma - 2 \log2 + \psi(\lambda) \right) \log\frac1{1-z} + C\right)\\
    &+
    \mathcal{O}\left((1-z)\left(\log(1-z)\right)^2\right),
  \end{split}
\end{equation}
where $C$ is an explicit constant that does not influence the later results.

The method of singularity analysis (see
\cite{Flajolet_Odlyzko1990:singularity_analysis_generating}), explained in
Section \ref{sec:singularity-analysis} and applied in the proof of Theorem~\ref{thm:jacobi}, allows to
translate \eqref{eq:Jspecial-asymp} into an asymptotic expressions for
$J_n^{(\lambda;\lambda-\frac12)}$ as given in the result. 
\end{proof}

\section{Special cases}\label{sec:special-cases}

In this section we consider connection formulas and relations between the parameters $\lambda$, $\alpha$, $\beta$, and $\mu$ not covered by the generic case.

\subsection{General connection formulas}

A direct corollary of \cite[Theorem~3]{SanchezRuiz2001:linearization_and_connection_formula} is the following result connecting $I_n^{(\lambda; \alpha, \beta)}$ and $I_k^{(\rho; \alpha, \beta)}$. 
\begin{prop} Let $\alpha, \beta > - 1$ and $\lambda, \rho > 0$. Then
\begin{equation*}
\begin{split}
I_n^{(\lambda; \alpha, \beta)} 
&= \frac{\Pochhsymb{2\lambda}{n}}{(n!)^2} \sum_{k=0}^n \binom{n}{k} \frac{(k!)^2 \Pochhsymb{k+2\lambda}{n} \Pochhsymb{\lambda}{k} \Pochhsymb{\rho + \frac{1}{2}}{k}}{\Pochhsymb{2\rho}{2k} \Pochhsymb{\rho}{k} \Pochhsymb{\lambda + \frac{1}{2}}{k}} \\
&\phantom{equals}\times \Hypergeom{5}{4}{k-n,k+n+2\lambda,k+\lambda,k+2\rho,k+\rho+\frac{1}{2}}{2k+2\rho+1,k+\rho,k+2\lambda,k+\lambda+\frac{1}{2}}{1} \, I_k^{(\rho; \alpha, \beta)}.
\end{split}
\end{equation*}

\end{prop}

For the case $\beta - \alpha \in \mathbb{N}$, we have the following connection formula.
\begin{thm} \label{thm:connection.formula.alpha.beta}
Let $\alpha > -1$ and $\lambda > 0$. For $k \in \mathbb{N}$ with $k = 2m + \eta$ and $\eta \in \{0,1\}$, 
\begin{equation*}
I_n^{(\lambda; \alpha, \alpha + k)} = \sum_{\ell = 0}^m (-1)^\ell b_\ell \, J_n^{(\lambda;\ell+\alpha+\frac{1}{2})},
\end{equation*}
where  
\begin{equation*}
b_\ell = \sum_{\mu=\ell}^m \binom{2m+\eta}{2\mu} \binom{\mu}{\ell} = \binom{m}{\ell} \frac{\Pochhsymb{m+\eta}{m-\ell}}{\Pochhsymb{\frac{1}{2}+\eta}{m-\ell}} \times
\begin{cases}
1, & \eta = 0 \\
2 m + 1, & \eta = 1.
\end{cases}
\end{equation*}
\end{thm}

\begin{rmk}
By Theorem~\ref{thm:connection.formula.alpha.beta}, the asymptotic expansion of $I_n^{(\lambda; \alpha, \beta)}$ for $\beta - \alpha$ is a positive integer can be derived from the asymptotic expansions of $J_n^{(\lambda;\ell+\alpha+\frac{1}{2})}$, $\ell = 0, 1, \dots, \lfloor \frac{\beta-\alpha}{2} \rfloor$. For $\beta - \alpha$ is a negative integer, the roles of $\alpha$ and $\beta$ are interchanged.
\end{rmk}

\begin{proof}[Proof of Theorem~\ref{thm:connection.formula.alpha.beta}]
Using the definition of $I_n^{(\lambda; \alpha, \beta)}$, we write
\begin{align*}
I_n^{(\lambda; \alpha, \alpha + k)} 
&= \int_{-1}^1 \left( 1 - x \right)^k \left( 1 - x^2 \right)^\alpha \left( \GegenbauerC_n^{(\lambda)}( x ) \right)^2 \dd x \\
&= \sum_{\ell=0}^k \binom{k}{\ell} \int_{-1}^1 x^\ell \left( 1 - x^2 \right)^\alpha \left( \GegenbauerC_n^{(\lambda)}( x ) \right)^2 \dd x.
\end{align*}
The result follows by observing that for odd integers $\ell$ the integral vanishes and 
\begin{equation*}
\left( x^2 \right)^\mu = \left( 1 - \left( 1 - x^2 \right) \right)^\mu = \sum_{\ell=0}^\mu \binom{\mu}{\ell} (-1)^\ell \left( 1 - x^2 \right)^\ell.
\end{equation*}

\end{proof}

\subsection{The case $\lambda-\mu\in\mathbb{N}$}

\begin{thm}\label{thm:lambda-k}
  Let $\lambda>0$ and $k\in\mathbb{N}$ with $k\leq\lambda$. Then
  $\mathcal{J}^{(\lambda;\lambda-k)}(z)$ is a rational function with
  denominator $(1-z)^{2k-1}$. As a consequence
  \begin{equation*}
  J_n^{(\lambda;\lambda-k)}=\frac{(2\lambda)_n}{n!}q_k(n),
\end{equation*}
where $q_k$ is a polynomial of degree $2k-2$. The generating function and the
explicit formula for the coefficients are given in \eqref{eq:J-exact} and
\eqref{eq:Jn-exact}. The main term of the asymptotics is given in \eqref{eq:Jn-asymptotics}.
\end{thm}

\begin{proof}
We first show that for $k\geq1$ we obtain a rational generating function
\begin{equation*}
  \mathcal{J}^{(\lambda;\lambda-k)}(z)=
  \frac{p_k(z)}{(1-z)^{2k-1}},
\end{equation*}
where $p_k$ is a polynomial of degree $2k-2$.

This can be seen by considering the differential equation satisfied by
$\mathcal{J}^{(\lambda;\lambda-k)}(z)(1-z)^{2k-1}$
\begin{equation} \label{eq:Kmod}
  \begin{split}
    z^2&(z-1)^2y'''+z(z-1) ((3\lambda-5k+7) z - 3 \lambda +k-2) y''\\
    &+ 2 \biggl( (1+\lambda -k) (\lambda -4 k+5) z^2 - \Big( (1+\lambda -k) (2\lambda -4 k+5) - 2 (k-1)^2 \Big) z \\
      &\phantom{+ 2 \biggl(}+ \lambda  ( 1 + \lambda  - k) \biggr)y'-2(k-1)(2\lambda-2k+1)(( 1 + \lambda-k)z-\lambda)y=0.
  \end{split}
\end{equation}

It translates into the three-term recurrence relation

%

\begin{equation}\label{pn-recurr}
  \begin{split}
    &(2 - 2 k + n) (1 + 2 \lambda - 2 k + n) (1 + \lambda - k + n)p_n \\
    &-\Bigl( \left( 2 - 2k + n \right) \big( \left( 2 - 2k + n \right) \left( 4 + 3 \lambda -k + 2n \right) + \lambda + k + 4 \lambda k \big) + 2 \lambda k \Bigr)p_{n+1} \\
    &+ (3 - 2 k + n)^2 (2 + n)p_{n+2}=0
  \end{split}
  \end{equation}
  for the coefficients of
  \begin{equation*}
    p(z)=\sum_{n=0}^\infty p_n(1-z)^n.
  \end{equation*}

  Specialising $n=2k-2$ and $n=2k-3$ gives
\begin{align*}
   p_{2k} &=  \lambda \, p_{2k-1}, \\
   (k-1) \, p_{2k-2} &=  (\lambda + k-2) \, p_{2k-3}
  \end{align*}
which shows that $p_{2k-1}$ can be chosen as $0$, without influencing the
coefficients $p_0,\ldots,p_{2k-2}$. Then $p_\ell=0$ for $\ell\geq2k-1$. Thus
\eqref{eq:Kmod} has a polynomial solution, which has to coincide with the only
solution holomorphic around $z=0$.

In order to compute the coefficients in the partial fraction decomposition
\begin{equation*}
  \mathcal{J}^{(\lambda;\lambda-k)}(z)=\sum_{\ell=1}^{2k-1}\frac{c_\ell}{(1-z)^\ell}
\end{equation*}
we use \eqref{eq:J-mellin-barnes} and shift the line of integration to the left
to line $\Re(s)=-\lambda$; this time we choose the contour to encircle
$s=-\lambda$ to the right. Collecting residues gives
\begin{equation}\label{eq:J-residues}
\begin{split}
  \mathcal{J}^{\lambda;\lambda-k}(z)&=
   \frac{\sqrt\pi\Gamma(2\lambda)}{\Gamma(\lambda)^2(1-z)^{2\lambda}}
  \frac1{2\pi i}\int\limits_{-\lambda-i\infty}^{-\lambda+i\infty}
  \frac{\Gamma(s+\lambda)^2\Gamma(s+\lambda-k+\frac12)\Gamma(-s)}
  {\Gamma(s+2\lambda)\Gamma(s+\lambda-k+1)}\left(\frac{4z}{(1-z)^2}\right)^s
  \dd s\\
  &+
  \frac{\sqrt\pi\Gamma(2\lambda)}{\Gamma(\lambda)^2(1-z)^{2\lambda}}
  \sum_{\ell=0}^{k-1}\frac{(-1)^\ell}{\ell!}
  \frac{\Gamma(k-\ell-\frac12)^2\Gamma(\lambda-k+\ell+\frac12)}
  {\Gamma(\lambda+k-\ell-\frac12)\Gamma(-\ell+\frac12)}
  \left(\frac{4z}{(1-z)^2}\right)^{-\lambda+k-\ell-\frac12}.
\end{split}
\end{equation}

The sum of residues simplifies to
\begin{equation}\label{eq:sum-res}
  \frac{\Gamma(2\lambda)}{\Gamma(\lambda)^2}
  \sum_{\ell=0}^{k-1}\frac{(\frac12)_\ell}{\ell!}
  \frac{\Gamma(k-\ell-\frac12)^2\Gamma(\lambda-k+\ell+\frac12)}
  {\Gamma(\lambda+k-\ell-\frac12)}
  (4z)^{-\lambda+k-\ell-\frac12}(1-z)^{2\ell-2k+1}.
\end{equation}
We already know that $\mathcal{J}^{(\lambda;\lambda-k)}(z)$ is a rational
function with denominator $(1-z)^{2k-1}$. Furthermore, the integral in
\eqref{eq:J-residues} behaves like $\mathcal{O}(1)$ for $z\to1$. Thus, the
rational function and the coefficients in its partial fraction decomposition
can be obtained from the corresponding asymptotic terms of the sum
\eqref{eq:sum-res} for $z\to1$. For this purpose we rewrite the powers of $4z$
in terms of the binomial series to obtain
\begin{equation}\label{eq:J-exact}
  \mathcal{J}^{(\lambda;\lambda-k)}(z)
  =\frac{4^{k-1} \sqrt{\pi} \, \Gamma(\lambda+\frac12)}
  {\Gamma(\lambda)(1-z)^{2k-1}}
  \sum_{r=0}^{2k-2}(1-z)^r
  \sum_{\ell=0}^{\lfloor\frac r2\rfloor}\frac{(\frac12)_{k-\ell-1}^2\,(\frac12)_\ell}
  {4^\ell\,(r-2\ell)!\,\ell!\,(\lambda-k-\ell+r+\frac12)_{2k-r-1}}.
\end{equation}
From this we can read off an exact formula for the coefficients
\begin{equation}\label{eq:Jn-exact}
  \begin{split}
    J_n^{(\lambda;\lambda-k)}&= \frac{(2\lambda)_n}{n!} \frac{4^{k-1}
      \sqrt{\pi}\,\Gamma(\lambda+\frac12)}{\Gamma(\lambda)}
    \sum_{r=0}^{2k-2}\binom{n+2k-2-r}{n}\\
    &\times \sum_{\ell=0}^{\lfloor\frac
      r2\rfloor}\frac{(\frac12)_{k-\ell-1}^2\,(\frac12)_\ell}
    {4^\ell\,(r-2\ell)!\,\ell!\,(\lambda-k-\ell+r+\frac12)_{2k-r-1}}.
  \end{split}
\end{equation}
The asymptotic main term is given by
\begin{equation} \label{eq:Jn-asymptotics}
  J_n^{(\lambda;\lambda-k)} = \frac{\sqrt\pi \, \gammafcn( k - \frac{1}{2} ) \gammafcn( \frac{1}{2} + \lambda - k )}
  {2^{2\lambda-1}\Gamma(\lambda)^2 (k-1)! \gammafcn( k + \lambda - \frac{1}{2} )} \, n^{2\lambda+2k-3}+
  \mathcal{O}(n^{2\lambda+2k-4}).
\end{equation}
\end{proof}

\subsection{The case $\mu-\lambda\in\mathbb{N}$}
In this case we observe that \eqref{eq:Jn-connect} has at most $k\DEF\mu-\lambda$
terms and there occur some extra cancellations in the Pochhammer-symbols. We get in particular  
\begin{equation*}
J_n^{(\lambda;\lambda)} = \frac{\Pochhsymb{2\lambda}{n}}{n!} \, \frac{\sqrt{\pi} \, \gammafcn( \lambda + \frac{1}{2} )}{\gammafcn( \lambda )} \, \frac{1}{n + \lambda}
\end{equation*}
and for $k \in \mathbb{N}$:
\begin{align*}
  J_n^{(\lambda;\lambda+k)} &= \frac{2\pi}{4^{\lambda+k}\Gamma(\lambda)^2}
  \sum_{\ell=0}^{\min(k,\lfloor \frac{n}{2} \rfloor)} {\binom k \ell}^2 \left(
    \frac{\gammafcn(n-\ell+\lambda)}{\gammafcn(n+k-\ell+1+\lambda)}\right)^2\\
  &\times\left( n - 2\ell + k + \lambda \right) \frac{\gammafcn( n + 2k - 2\ell + 2\lambda )}{\gammafcn( n - 2 \ell + 1 )}.
\end{align*}
All terms in the sum have the same asymptotic order as $n\to\infty$. We use
\begin{equation*}
  \sum_{\ell=0}^k{\binom k \ell}^2=\binom{2k}k = \frac{2^{2k} \gammafcn( k + \frac{1}{2} )}{\sqrt{\pi} \, \gammafcn( k + 1 )}
\end{equation*}
to obtain the following asymptotic formula.
\begin{thm}\label{thm:mu=lambda+m}
  Let $\lambda>0$ and $k\in\mathbb{N}$. Then
  \begin{equation*}
  J_n^{(\lambda;\lambda+k)}=\frac{2\pi}{4^{\lambda+k}\Gamma(\lambda)^2} \binom{2k}k \,  n^{2\lambda-2} + \mathcal{O}(n^{2\lambda-3}) \qquad \text{as $n \to \infty$.}
\end{equation*}
\end{thm}
Of course, a full asymptotic expansion could be given with more effort.

\subsection{The case $\lambda \in \mathbb{N}$ and $\mu, 2\mu \notin \mathbb{Z}$}

Let $\lambda = k$, $k \in \mathbb{N}$. The case $\mu \in \mathbb{Z}$ is covered above. The case of $\mu$ is a half-integer is more involved and not done here. 
The setting is a non-generic case for the Gegenbauer weight. To obtain the local expansion around $z = 1$, we proceed as before by using \eqref{eq:J-mellin-barnes} and shifting the line of integration to the left. Collecting the residues at the double poles in $-k-\ell$, $0 \leq \ell \leq k - 1$, we obtain
\begin{equation*}
\begin{split}
&\frac{\sqrt{\pi} \, \gammafcn( \mu + \frac{1}{2} )}{\gammafcn( \mu + 1 )} \, \frac{\Pochhsymb{\frac{1}{2}}{k} \Pochhsymb{-\mu}{k}}{\gammafcn( k ) \Pochhsymb{\frac{1}{2}-\mu}{k}} \, z^{-k} \Hypergeom{3}{2}{1-k,k,k-\mu}{1,k + \frac{1}{2} - \mu}{- \frac{(1-z)^2}{4z}} \, \log \frac{1}{1-z} \\
&\phantom{equals}+ \text{power series in $(1-z)$.}
\end{split}
\end{equation*}
Due to cancellation, the integrand in \eqref{eq:J-mellin-barnes} has simple poles in $-k-\ell$ for $\ell \geq k$. Collecting those residues, we obtain
\begin{equation*}
\frac{\sqrt{\pi} \, \gammafcn( \mu + \frac{1}{2} )}{4 \gammafcn( \mu + 1 )} \, \frac{\Pochhsymb{\frac{1}{2}}{k}^2 \Pochhsymb{-\mu}{2k}}{k! k! \Pochhsymb{\frac{1}{2}-\mu}{k}} \frac{(1-z)^{2k}}{z^{2k}} \Hypergeom{4}{3}{1,1,2k,2k-\mu}{k+1,k+1,2k+\frac{1}{2}-\mu}{-\frac{(1-z)^2}{4z}}
\end{equation*}
which is holomorphic about $z = 1$ and as a power series in $(1-z)$ will, thus, not play a role in the singularity analysis. Collecting residues at the simple poles in $-\mu-\frac{1}{2}-\ell$, $\ell \in \mathbb{N}_0$, we have
\begin{equation*}
\begin{split}
&2^{2k-2\mu-2} \gammafcn( - \frac{1}{2} - \mu ) \gammafcn( \mu + \frac{1}{2} ) \frac{\Pochhsymb{\frac{1}{2}}{k} \Pochhsymb{- \frac{1}{2} - \mu}{k}^2}{\gammafcn( k ) \Pochhsymb{- \frac{1}{2} - \mu}{2k}} \\
&\phantom{equals}\times \left( 1 - z \right)^{1+2\mu-2k} z^{-\frac{1}{2}-\mu} \Hypergeom{3}{2}{\frac{1}{2}, \frac{1}{2} + \mu, \frac{3}{2} + \mu - 2k}{\frac{3}{2} + \mu - k, \frac{3}{2} + \mu - k}{-\frac{(1-z)^2}{4z}}.
\end{split}
\end{equation*}
Thus, we arrive at
\begin{equation*}
\mathcal{J}^{(k;\mu)}(z) = \sum_{m=0}^\infty A_m \left( 1 - z \right)^m \log \frac{1}{1-z} + \sum_{m=0}^\infty B_m \left( 1 - z \right)^{1+2\mu-2k+m} + \text{power series in $(1-z)$},
\end{equation*}
where 
\begin{align*}
A_m &= \frac{\sqrt{\pi} \, \gammafcn( \mu + \frac{1}{2} )}{\gammafcn( \mu + 1 )} \, \frac{\Pochhsymb{\frac{1}{2}}{k} \Pochhsymb{-\mu}{k}}{\gammafcn( k ) \Pochhsymb{\frac{1}{2}-\mu}{k}} \sum_{\ell=0}^{\lfloor \frac{m}{2} \rfloor} \frac{\Pochhsymb{1-k}{\ell} \Pochhsymb{k-\mu}{\ell} \Pochhsymb{k}{m-\ell}}{\Pochhsymb{k+\frac{1}{2}-\mu}{\ell} (\ell!)^2 (m-2\ell)!} \, \frac{(-1)^\ell}{4^\ell}, \\
B_m &= 2^{2k-2\mu-2} \gammafcn( - \frac{1}{2} - \mu ) \gammafcn( \mu +
\frac{1}{2} ) \frac{\Pochhsymb{\frac{1}{2}}{k} \Pochhsymb{- \frac{1}{2} -
    \mu}{k}^2}{\gammafcn( k ) \Pochhsymb{- \frac{1}{2} - \mu}{2k}}\\
&\times\sum_{\ell=0}^{\lfloor \frac{m}{2} \rfloor} \frac{\Pochhsymb{\frac{1}{2}}{\ell} \Pochhsymb{\frac{3}{2} + \mu - 2 k}{\ell} \Pochhsymb{\mu + \frac{1}{2}}{m-\ell}}{\Pochhsymb{\frac{3}{2}+\mu-k}{\ell} \Pochhsymb{\frac{3}{2}+\mu-k}{\ell} \ell! (m - 2\ell)!} \, \frac{(-1)^\ell}{4^\ell}.
\end{align*}
Observe, that the sum in the expression for $A_m$ has at most $k-1$ terms. Using singularity analysis, this translates into the following asymptotic series:
\begin{equation*}
J_n^{(k;\mu)} = \frac{(2k)_n}{n!} \left( \sum_{m=0}^\infty (-1)^m \frac{m!}{n(n-1)\cdots(n-m)} A_m + \sum_{m=0}^\infty \binom{n+2k-2\mu-2-m}{n} B_m  \right).
\end{equation*}

\subsection{Other non-generic cases} 

In principle, our method can be used to get the full asymptotic expansion. Due to pole cancellation and pole multiplication in the integrands of \eqref{eq:I-mellin-barnes} and \eqref{eq:J-mellin-barnes} depending on assumptions on interrelations between the parameters $\lambda$, $\alpha$, $\beta$, and $\mu$, and the position of the line of integration as it is moved to the left, computations are rather involved. For example, in the case of $\lambda \in \mathbb{N}$ and $\mu$ a positive half-integer such that $0 < m = \mu + \frac{1}{2} < \lambda = k$, the integrand in \eqref{eq:J-mellin-barnes} has in $-m-\ell$: simple poles for $0 \leq \ell < k - m$, triple poles for $k - m \leq \ell < 2k-m$, and double poles for $\ell \geq 2k-m$.

\begin{ackno}
  The authors are grateful to Peter Paule for providing them with his 
  \texttt{Mathematica} package, which implements Zeilberger's algorithm and
  allows for transforming differential equations into holonomic linear
  recurrences for the coefficients. They are also indepted to Helmut Prodinger
  for pointing out to them that this method could be applied.
\end{ackno}
\bibliographystyle{abbrv}
\bibliography{REFS}

\end{document}